\documentclass[11pt]{article}
\usepackage[a4paper,margin=1in]{geometry}
\usepackage[T1]{fontenc}
\usepackage[utf8]{inputenc}
\usepackage{lmodern}
\usepackage{amsmath,amssymb,amsthm,mathtools,mathrsfs}
\usepackage{booktabs,array,tabularx,enumitem}
\usepackage{caption}
\usepackage[numbers,sort&compress]{natbib}
\usepackage{microtype}
\usepackage[hidelinks]{hyperref}
\usepackage[nameinlink,capitalize,noabbrev]{cleveref}
\usepackage{xurl}
\allowdisplaybreaks
\renewcommand{\arraystretch}{1.12}
\setlist[itemize]{leftmargin=1.5em,itemsep=2pt,topsep=3pt}
\newtheorem{theorem}{Theorem}[section]
\newtheorem{lemma}[theorem]{Lemma}
\newtheorem{proposition}[theorem]{Proposition}
\newtheorem{corollary}[theorem]{Corollary}
\theoremstyle{definition}

\theoremstyle{remark}

\newcommand{\E}{\mathbb E}
\newcommand{\Pp}{\mathbb P}
\newcommand{\R}{\mathbb R}
\newcommand{\ip}[2]{\langle #1,#2\rangle}
\newcommand{\norm}[1]{\left\|#1\right\|}
\newcommand{\KL}{\operatorname{KL}}

\newcommand{\prox}{\operatorname{prox}}
\newcommand{\cH}{\mathcal H}
\newcommand{\cE}{\mathcal E}
\newcommand{\ind}{\mathbf 1}
\newcommand{\Ham}{\operatorname{Ham}}

\newcommand{\authornames}{Haihan Zhang, Wendao Wu, Chenheng Zhang, Yanyi Li, Chunyuan Zheng, Cong Fang, Haoxuan Li, Zhouchen Lin}
\newcommand{\affiliation}{Peking University}
\newcommand{\emailA}{colombo.2026@outlook.com}

\newcommand{\mail}[1]{\href{mailto:#1}{\texttt{#1}}}
\newif\ifanonymous
\anonymousfalse
\ifanonymous
\author{Anonymous Authors}
\hypersetup{pdfauthor={Anonymous Authors}}
\else
\author{Haihan Zhang$^{1,*}$\quad Wendao Wu$^{1,*}$\quad Chenheng Zhang$^{1,*}$\\[0.20em]
Yanyi Li$^1$\quad Chunyuan Zheng$^1$\\[0.20em]
Cong Fang$^{1,\dagger}$\quad Haoxuan Li$^{1,\dagger}$\quad Zhouchen Lin$^{1,\dagger}$\\[0.60em]
\small $^1$\affiliation\\[0.35em]
\footnotesize\mail{\emailA} \\[0.35em]
\small $^*$Equal contribution.\quad $^\dagger$Corresponding authors.}
\hypersetup{pdfauthor={\authornames}}
\fi
\newcommand{\ResearchAgentSystem}{our laboratory's internal auto-research system}

\title{\textbf{Sharp High-Probability Statistical Rates for Heavy-Tailed Nonsmooth Convex Optimization}}
\date{}
\hypersetup{pdftitle={Sharp High-Probability Statistical Rates for Heavy-Tailed Nonsmooth Convex Optimization}}
\begin{document}
\maketitle

\begin{abstract}
We close the gap between known high-probability upper and lower bounds for both general convex and strongly convex nonsmooth stochastic optimization under heavy-tailed noise. Our upper bounds separate the cost of confidence from the total noise magnitude, attaining the statistical lower-bound orders established by \citet{liu2026}. A fresh stochastic subgradient has centered norm and directional $p$th moments bounded by $L^p$ and $s^p$, respectively, with $1<p\le2$ and $0<s\le L$. For $T$ oracle calls and failure probability $\delta$, write $q=1-1/p$, $\ell=\log(3/\delta)$, and $A=s(L^2/s^2)^q$. With initial distance bound $D$ and strong-convexity modulus $\mu$, the minimax objective errors in the nonsaturated statistical regime are
\[
\begin{aligned}
\text{General convex:}\qquad &\Theta\!\left(\frac{D(A+s\ell^q)}{T^q}\right),\\[3pt]
\text{Strongly convex:}\qquad &\Theta\!\left(\frac{A^2+s^2\ell^{2q}}{\mu T^{2q}}\right).
\end{aligned}
\]
The second rate applies in the quadratic regime; matching requires dimension $d\ge\lceil L^2/s^2\rceil$ and domination of the deterministic and finite-variance terms. The upper bounds hold in every finite dimension. They replace the previous confidence coefficient $s^{1/p}L^{1-1/p}$ by $s$, and its square by $s^2$ under strong convexity. The key is an online inequality that controls directional fluctuations separately from the total second moment supplied by rejection sampling. Localization preserves this separation under strong convexity without additional horizon logarithms. The result applies to known proximal regularizers and permits unrestricted internal computation.
\end{abstract}
\noindent\textbf{Keywords:} stochastic convex optimization; heavy-tailed noise; high-probability bounds; effective dimension; oracle complexity.

\begin{samepage}
\begingroup\frenchspacing\looseness=-1
\noindent\textbf{AI Usage.}\ 
Nearly the entire research pipeline for this paper was carried out by Colombo, 
\ResearchAgentSystem{}, powered by GPT-5.6 Sol. The system also conducted a
Lean-backed article audit of the resulting manuscript. The authors subsequently reviewed and approved the
mathematical claims, presentation, and formal artifacts, and take
responsibility for the final manuscript. The complete Lean audit report and the system's technical report will be made public at a later date.\par
\endgroup
\end{samepage}

\paragraph{Publication Intent and Recognition of Independent Work.} 
This preprint is released as a research output of our study of
automated scientific research, for public scrutiny and discussion.
As we do not intend to submit this manuscript to a journal or
conference, it has not undergone the full editorial refinement
and internal review we would normally require for submission.
The authors remain responsible for its contents and for correcting
any errors or inaccurate attribution.

We recognize that releasing these results may affect researchers
who have independently been pursuing the same questions, including
work that is not yet public.
If this applies to your work, we invite you to contact us at
\texttt{colombo.2026@outlook.com}.
A brief description of the overlap, together with any dated
materials you are comfortable sharing, would help us understand
its scope and chronology.
Such materials might include draft excerpts, presentation slides,
or revision records; a complete manuscript or unpublished proof
is not required to begin the conversation.

We will review these cases carefully and update citations,
acknowledgments, and accounts of independent or concurrent work
as appropriate, including revising priority claims where warranted.
Nonpublic materials will be treated confidentially and used only
to assess the overlap unless their authors authorize otherwise.
We will obtain explicit permission from the authors before
publicly discussing such materials or providing them to an
automated research system.

\section{Introduction}\label{sec:introduction}

We study high-probability oracle complexity for nonsmooth general convex and strongly convex optimization when gradient noise has only a finite $p$th moment, for $1<p\le2$. Under a single bound on the norm moment, mirror descent and clipped stochastic-gradient methods attain the characteristic error orders $T^{-(1-1/p)}$ for convex problems and $T^{-2(1-1/p)}$ for strongly convex problems \citep{vural2022,liuzhou2023}. High-probability guarantees require more than the correct exponent of the sample budget $T$: they must also identify how noise geometry interacts with the failure probability. A vector can have a large norm while every fixed projection has a much smaller moment, so a norm bound alone can obscure the cost of confidence.

We study this distinction for composite objectives $F=f+r$ on a known closed convex set. The unknown function $f$ is convex and $G$-Lipschitz; the known regularizer $r$ admits proximal evaluations and is $\mu$-strongly convex, where $\mu=0$ includes arbitrary convex regularizers and $r=0$. The initial point is within distance $D$ of a minimizer. One charged oracle call returns a fresh unbiased subgradient estimate of $f$; it does not reveal a function value or a reusable sample function. For its centered noise $\zeta$, assume
\begin{equation}\label{eq:intro-moments}
 \E\norm\zeta^p\le L^p,
 \qquad \sup_{\norm v=1}\E|\ip v\zeta|^p\le s^p,
 \qquad 0<s\le L.
\end{equation}
The conditions hold conditionally at adaptive queries. The ratio $e=L^2/s^2$ measures the separation between total and directional noise and serves as a generalized effective dimension \citep{liu2026}.

\paragraph{The unresolved confidence dependence.}
Set $q=1-1/p$, $\ell=\log(3/\delta)$ for failure probability $0<\delta<1/10$, and
\[
 A=s e^q=s^{2/p-1}L^{2-2/p},
 \qquad B=s^{1/p}L^{1-1/p}=s e^{q/2}.
\]
In the heavy-tail statistical regime, the refined clipped-gradient analysis of \citet[Theorems 1--2]{liu2026} gives upper bounds of order $D(A+B\ell^q)/T^q$ and $(A^2+B^2\ell^{2q})/(\mu T^{2q})$ for the convex and strongly convex cases. Its high-probability lower bounds have $s$ and $s^2$ in the respective confidence terms. The gap is polynomial in $e$: $B/s=e^{q/2}$. Thus the existing bounds leave open whether the larger confidence cost is intrinsic to optimization or can be removed by a different method.

\paragraph{Our result.}
We attain the lower confidence scale in both curvature regimes. Let $\cE_T^{(\mu)}(\delta)$ denote the smallest objective-error threshold guaranteed with probability $1-\delta$, uniformly over the model, by an adaptive randomized algorithm with at most $T$ raw oracle calls. In the nonsaturated heavy-tail regime, our results give
\begin{equation}\label{eq:intro-rates}
 \cE_T^{(0)}(\delta)\asymp_{p,C_0}
 \frac{D(A+s\ell^q)}{T^q},
 \qquad
 \cE_T^{(\mu)}(\delta)\asymp_{p,C_0}
 \frac{A^2+s^2\ell^{2q}}{\mu T^{2q}}\quad(\mu>0),
\end{equation}
where the second expression applies in the quadratic branch of the strongly convex error. The upper bounds hold in every finite dimension; matching lower bounds require $d\ge\lceil e\rceil$. Here $C_0$ is the fixed domination constant in \cref{cor:heavy}; subscripts on $\asymp$ specify the allowed dependence of its positive constants. The precise conditions under which the statistical terms dominate are given in \cref{cor:matching,cor:heavy}.

\Cref{tab:comparison} summarizes the improvement and distinguishes the assumptions of earlier benchmarks.

\begin{table}[t]
\centering
\small
\caption{High-probability objective-error bounds after $T$ fresh stochastic-gradient calls, retaining the leading nonsaturated statistical terms. Here $q=1-1/p$, $\ell=\log(3/\delta)$, $e=L^2/s^2$, $A=se^q$, $B=se^{q/2}$, and $h_T=\log(\log(\mathrm e T)/\delta)$. Omitted terms must be dominated. In particular, the Liu lower-bound entries require $T\gtrsim_p1+\ell/e$ and inactive amplitude caps ($G$ in the convex case; $G$ and $\mu D$ in the strongly convex case). The Das et al. row assumes a fixed covariance envelope $\Sigma$ with $\operatorname{tr}\Sigma\le L^2$, $\|\Sigma\|_{\rm op}\le s^2$, and $T\ge\log\log d$, in the regime where its additional finite-budget terms are dominated. Liu--Zhou study a Lipschitz objective on a projected domain. Matching orders require $d\ge\lceil e\rceil$ and the conditions of \cref{cor:heavy}, which defines $C_0$.}
\label{tab:comparison}
\begingroup
\renewcommand{\arraystretch}{1.45}
\setlength{\tabcolsep}{4pt}
\begin{tabularx}{\linewidth}{@{}>{\raggedright\arraybackslash}p{0.20\linewidth}>{\centering\arraybackslash}X>{\centering\arraybackslash}X>{\raggedright\arraybackslash}p{0.23\linewidth}@{}}
\toprule
Work & Upper bound & Lower bound & Regime / scope\\
\midrule
\multicolumn{4}{@{}l}{\textbf{General convex} ($\mu=0$)}\\
\addlinespace[2pt]
\citet{liuzhou2023}\newline (Theorem 3)
& $\displaystyle O_p\!\left(\frac{DL\ell}{T^q}\right)$ & ---
& Norm moment only; projected method\\
\addlinespace[4pt]
\citet{das2024}\newline (Theorem 4)
& $\displaystyle O\!\left(D\sqrt{\frac{L^2+sLh_T}{T}}\right)$ & ---
& $p=2$, $G\le L$; covariance envelope\\
\addlinespace[4pt]
\citet{liu2026}
& $\displaystyle O_p\!\left(\frac{D(A+B\ell^q)}{T^q}\right)$
& $\displaystyle\Omega_p\!\left(\frac{D(A+s\ell^q)}{T^q}\right)$
& Two moments; confidence gap\\
\addlinespace[4pt]
\textbf{This Paper}\newline (\cref{cor:heavy})
& \multicolumn{2}{c}{$\displaystyle\boldsymbol{\Theta_{p,C_0}\!\left(\frac{D(A+s\ell^q)}{T^q}\right)}$}
& New upper bound matches known order\\
\midrule
\multicolumn{4}{@{}l}{\textbf{Strongly convex} ($\mu>0$, quadratic branch)}\\
\addlinespace[2pt]
\citet{liuzhou2023}\newline (Theorem 8)
& $\displaystyle O_p\!\left(\frac{L^2\ell^2}{\mu T^{2q}}\right)$ & ---
& Norm moment only; $G\le L$; Lipschitz loss\\
\addlinespace[4pt]
\citet{liu2026}
& $\displaystyle O_p\!\left(\frac{A^2+B^2\ell^{2q}}{\mu T^{2q}}\right)$
& $\displaystyle\Omega_p\!\left(\frac{A^2+s^2\ell^{2q}}{\mu T^{2q}}\right)$
& Two moments; confidence gap\\
\addlinespace[4pt]
\textbf{This Paper}\newline (\cref{cor:heavy})
& \multicolumn{2}{c}{$\displaystyle\boldsymbol{\Theta_{p,C_0}\!\left(\frac{A^2+s^2\ell^{2q}}{\mu T^{2q}}\right)}$}
& New upper bound matches known order\\
\bottomrule
\end{tabularx}
\endgroup
\end{table}

\paragraph{The obstacle and the new inequality.}
The gap concerns which aspect of noise governs confidence. In the existing clipped-gradient analysis, a norm threshold limits large observations but also enters the concentration cost. Lowering that threshold reduces the fluctuation term while increasing bias; raising it reverses the tradeoff. Thus choosing a different threshold within that analysis does not by itself separate the two moment scales. Nor does an accurate estimate at a single query immediately supply an optimization method: successive queries are adaptive, and their mean subgradients need not agree.

Our main technical contribution is an online inequality whose leading confidence term uses only the directional moment. After preprocessing, let $V$ bound the total second moment and $s_0$ the directional $p$th-moment scale. The statistical part of its average regret over a radius-$R$ ball is of order $R[\sqrt{V/n}+s_0(\log(1/\delta)/n)^q]$. In particular, the total second-moment term is not multiplied by a confidence factor. The learner achieves this separation by integrating over Gaussian perturbations: bounded comparator directions use the directional moment, while Gaussian directions average to the total noise energy.

Conditional rejection then supplies a finite $V$ while controlling the induced bias through $s$. Combining the two estimates and choosing the rejection threshold yields $A/T^q+s(\ell/T)^q$. Every rejected draw counts toward the oracle budget. A proximal conversion realizes the online decisions as feasible queries, so the construction applies to the stated composite model.

\paragraph{General convex and strongly convex consequences.}
For general convex objectives, the common procedure directly attains the linear-in-$D$ error in \cref{eq:intro-rates}. For strongly convex objectives, decreasing-radius localization converts the same guarantee into the quadratic statistical rate. Allocating failure probabilities according to stage cost avoids an additional $\log T$ or $\log\log T$ loss. Thus both curvature regimes retain the smaller directional confidence scale. The finite-budget lower bound complements these upper bounds through a common construction that retains the Gaussian term and the transition imposed by initialization; its leading heavy-tail orders agree with \citet{liu2026}.

\paragraph{Scope and organization.}
The contribution is a matching characterization of the statistical terms, together with explicit finite-budget upper and lower bounds. The algorithm knows the moment and geometry parameters and may perform unrestricted internal computation; its Gaussian-mixture representation can be exponentially large in dimension. We do not claim computational efficiency or matching deterministic transients in every fixed dimension. \Cref{sec:model} states the model and unified theorem, \cref{sec:upper-overview,sec:lower-overview} explain the upper and lower mechanisms, and \cref{sec:related} places the results in the literature. Complete proofs are included in the appendix.

\section{Model and Main Results}\label{sec:model}

\subsection{Composite objectives and fresh stochastic subgradients}
Fix $1<p\le2$, $G,D>0$, $0<s\le L$, $\mu\ge0$, and a finite dimension $d$. An instance consists of a known nonempty closed convex set $\mathcal X\subseteq\R^d$, an unknown convex function $f$ that is $G$-Lipschitz on $\mathcal X$, and a known convex regularizer $r$ that is $\mu$-strongly convex on $\mathcal X$. Let $h=r+\iota_{\mathcal X}$ be proper, closed, and convex, where $\iota_{\mathcal X}$ is the indicator function of the set. We assume exact access to $\prox_{\eta h}$ for every $\eta>0$. A minimizer $x^*$ of $F=f+r$ exists, and the supplied feasible point $x_1$ satisfies $\norm{x_1-x^*}\le D$. The set and regularizer are known mathematical objects, and deterministic operations on them are uncharged. The domain may be unbounded, and the objective need not be smooth.

At each query $x\in\mathcal X$, the oracle returns a vector $g(x,\xi)$ from a fixed query-dependent kernel with mean in $\partial f(x)$. Samples are fresh and independent of the preceding history conditional on the query; repeated calls at a frozen query use independent copies of that kernel. If $\mathscr F_{t-1}$ is the history before the $t$th raw call and $\E_{t-1}=\E[\,\cdot\mid\mathscr F_{t-1}]$, the centered noise $\zeta_t=g_t-\E_{t-1}g_t$ obeys the two bounds
\begin{equation}\label{eq:native}
 \E_{t-1}\norm{\zeta_t}^p\le L^p,
 \qquad \sup_{\norm v=1}\E_{t-1}|\ip v{\zeta_t}|^p\le s^p.
\end{equation}
Every raw observation is charged, including observations discarded by rejection sampling. Function values and access to reusable random functions are unavailable. Algorithms may choose queries adaptively, randomize internally, and return any feasible point. The public inputs include $G,D,L,s,p,\mu,T,\delta$ and the known set and regularizer. Our information bound imposes no restriction on internal computational cost.

Let $\mathfrak I_\mu$ be this class of instances and $\mathcal A_T$ the algorithms that use at most $T$ raw calls almost surely. With $F_I^*=\min F_I$, define
\begin{equation}\label{eq:minimax}
 \cE_T^{(\mu)}(\delta)=
 \inf_{\mathcal A\in\mathcal A_T}\inf\left\{
 \epsilon\ge0:\sup_{I\in\mathfrak I_\mu}
 \Pp_I\{F_I(\widehat x_{\mathcal A})-F_I^*>\epsilon\}\le\delta
 \right\}.
\end{equation}
Probability includes oracle and algorithm randomness. Dependence on the other public parameters and $d$ is suppressed. We write $\cE_T^{\rm cvx}=\cE_T^{(0)}$ and $\cE_T^{\rm sc}=\cE_T^{(\mu)}$ when $\mu>0$.

\subsection{Sharp rates in the two curvature regimes}
Throughout the paper, let
\begin{equation}\label{eq:notation}
 q=1-\frac1p,\quad r_p=\frac1q,\quad
 e=\frac{L^2}{s^2},\quad A=se^q,\quad
 \ell=\log\frac3\delta,\quad K=e+\ell.
\end{equation}
The moment ratio $e$ need not equal the ambient dimension. We first state the sharp statistical rates, then give the full finite-budget bounds.

\begin{theorem}[Sharp statistical rates]\label{cor:heavy}
Let $u_T=s(K/T)^q$. Suppose $d\ge\lceil e\rceil$, $u_T\le G$, and $G/\sqrt T+G\ell/T+L/\sqrt T\le C_0u_T$ for a fixed $C_0$. Then
\[
 \cE_T^{\rm cvx}(\delta)\asymp_{p,C_0}
 \frac{D(A+s\ell^q)}{T^q}.
\]
If~also $\mu>0$ and $u_T\le\mu D$, then
\[
 \cE_T^{\rm sc}(\delta)\asymp_{p,C_0}
 \frac{A^2+s^2\ell^{2q}}{\mu T^{2q}}.
\]
\end{theorem}
\begin{proof}
Use $s(e+\ell)^q\asymp_p A+s\ell^q$, \cref{eq:Hscale}, and the relevant branch of \cref{eq:H} in \cref{thm:main}.
\end{proof}

\subsection{Finite-budget bounds and curvature transitions}
The statistical theorem isolates the main contribution. To retain the finite-variance term and describe initialization and saturation, define
\begin{equation}\label{eq:intro-envelope}
 w_T=\frac L{\sqrt T}+s\left(\frac{e+\ell}{T}\right)^q,
 \qquad \cH_{\mu,D}(a)=\max_{0\le r\le D}\left\{ar-\frac\mu2r^2\right\}.
\end{equation}
The envelope converts the common statistical amplitude $w_T$ into error in each curvature regime.

 We abbreviate $\cH_{\mu,D}$ by $\cH$. Its definition in \cref{eq:intro-envelope} gives
\begin{equation}\label{eq:H}
 \cH(a)=
 \begin{cases}
 Da,&\mu=0,\\
 a^2/(2\mu),&\mu>0,\ a\le\mu D,\\
 Da-\mu D^2/2,&\mu>0,\ a\ge\mu D.
 \end{cases}
\end{equation}
For~$\mu>0$, it lies between $\tfrac12\min\{Da,a^2/\mu\}$ and $\min\{Da,a^2/\mu\}$. For all $\mu\ge0$ and $0\le c\le1$,
\begin{equation}\label{eq:Hscale}
 c^2\cH(a)\le\cH(ca)\le c\cH(a).
\end{equation}
These elementary scaling inequalities let us combine the statistical lower bounds without losing the transition between the two curvature branches.

\begin{theorem}[Unified all-budget bounds]\label{thm:main}
For the model above, every integer $T\ge1$ and $0<\delta<1/10$, set
\begin{equation}\label{eq:new-v}
 w_T=\frac L{\sqrt T}+s\left(\frac KT\right)^q,
 \qquad v_T=\frac G{\sqrt T}+\frac{G\ell}{T}+w_T.
\end{equation}
There exist constants $c_p,C_p>0$, depending only on $p$, such that
\begin{equation}\label{eq:main}
 c_p\cH(\min\{G,w_T\})
 \le \cE_T^{(\mu)}(\delta)
 \le C_p\cH(\min\{G,v_T\}).
\end{equation}
The upper bound holds in every finite dimension. The lower bound holds for every $d\ge\lceil e\rceil$ and covers all adaptive randomized algorithms with the stipulated oracle access. Its instances have $\mathcal X=\R^d$ and the same known regularizer $r(x)=\mu\norm x^2/2$, including $r=0$ when $\mu=0$.
\end{theorem}

For general convex objectives, \cref{eq:main} is the sandwich $c_pD\min\{G,w_T\}\le\cE_T^{\rm cvx}(\delta)\le C_pD\min\{G,v_T\}$. Under strong convexity it gives, in particular,
\begin{equation}\label{eq:expanded}
 \cE_T^{\rm sc}(\delta)\le\frac{C_p}{\mu}
 \left[\frac{G^2+L^2}{T}+\frac{G^2\ell^2}{T^2}
       +\frac{A^2+s^2\ell^{2q}}{T^{2q}}\right].
\end{equation}
The envelope is sharper than this expanded expression when initialization or saturation dominates. The deterministic nonsmooth term $G/\sqrt T$ and the confidence transient $G\ell/T$ occur only in the upper amplitude. Thus \cref{thm:main} leaves the deterministic finite-budget complexity unresolved.

\begin{corollary}[Matching statistical regime]\label[corollary]{cor:matching}
Suppose $d\ge\lceil e\rceil$. If $w_T\ge G$, then $\cE_T^{(\mu)}(\delta)\asymp_p\cH(G)$. If, for a fixed $C_0>0$,
\begin{equation}\label{eq:match-cond}
 \frac G{\sqrt T}+\frac{G\ell}{T}\le C_0w_T,
\end{equation}
then $\cE_T^{(\mu)}(\delta)\asymp_{p,C_0}\cH(\min\{G,w_T\})$. A sufficient condition for \cref{eq:match-cond} with an absolute $C_0$ is $L\ge G$ and $T\ge(G/L)^2\ell^2$.
\end{corollary}
\begin{proof}
Monotonicity and \cref{eq:Hscale} compare the two sides of \cref{eq:main}. Under the sufficient condition, each transient is at most $L/\sqrt T$.
\end{proof}

In particular, if $\mu>0$, \cref{eq:match-cond} holds, and $w_T\le\min\{G,\mu D\}$, the full quadratic statistical rate is
\[
 \cE_T^{\rm sc}(\delta)\asymp_{p,C_0}
 \frac1\mu\left[\frac{L^2}{T}
       +\frac{A^2+s^2\ell^{2q}}{T^{2q}}\right].
\]
This retains the finite-variance term even when it is not dominated by the heavy-tail terms.

\Cref{tab:comparison} compares the closest two-moment bounds with earlier norm-moment and finite-variance benchmarks, marking differences in their assumptions. The improvement in the displayed statistical rates is in the upper bounds: they attain the leading lower-bound orders already established by \citet{liu2026}. Our lower bound also retains finite-budget terms and the initialization-dependent envelope in \cref{thm:main}, which are not displayed in this statistical comparison. The earlier expected-error bounds of \citet{vural2022} concern a different success criterion. The finite-variance results of \citet{parletta2024,gorbunov2024} and the improved anytime analysis of \citet{parletta2025} provide further context in \cref{sec:related}.

The heavy-tail rates make the effective-dimension gain explicit:
\begin{equation}\label{eq:effective}
 \frac{DL}{T^q}\left[e^{1/2-1/p}+\frac{\ell^q}{\sqrt e}\right],
 \qquad
 \frac{L^2}{\mu T^{2q}}\left[e^{1-2/p}+\frac{\ell^{2q}}e\right].
\end{equation}
Their confidence terms gain factors $e^{q/2}$ and $e^q$, respectively, relative to the prior upper bounds when confidence dominates. For $p=2$, the statistical amplitude is comparable to $(L+s\sqrt\ell)/\sqrt T$; its two moment scales retain distinct roles even at the finite-variance endpoint.

\paragraph{A common benchmark.}
Take $p=3/2$, $G=D=L=1$, $s=e^{-1/2}$, $\delta=\exp(-e)$, and $d\ge\lceil e\rceil$. For $T\ge C e^{7/4}$ and sufficiently large $e$, the convex rate is $\Theta(e^{-1/6}T^{-1/3})$ and, with $\mu=1$, the strongly convex rate is $\Theta(e^{-1/3}T^{-2/3})$. These follow from the same amplitude $u_T\asymp e^{-1/6}T^{-1/3}$; the confidence transient $e/T$ determines the displayed budget threshold. The resulting small-error sample complexities are $\Theta(e^{-1/2}\epsilon^{-3})$ and $\Theta(e^{-1/2}\epsilon^{-3/2})$, respectively. \Cref{app:consequences} gives the calculation and the general sample-complexity bounds.

\section{Upper Bound via Mixed-Moment Online Learning}\label{sec:upper-overview}

The central ingredient is a convex optimization procedure whose confidence cost depends on the directional moment. We first explain the online estimate it needs, then show how rejection and proximal access supply admissible observations. Strong convexity is used only after this common procedure has been established.

The procedure maintains virtual decisions, converts them into feasible proximal queries, and uses rejection sampling to obtain observations for the online update. It returns the average of the associated feasible points, counting every rejected draw toward a fixed oracle budget. One invocation handles general convex objectives; localization restarts the same procedure for strongly convex objectives. The following subsections explain why these steps preserve the two moment scales; \cref{app:online,sec:rejection} give the precise updates and parameter choices.

\subsection{An online inequality with two different moment controls}
Consider observations $Y_t=b_t+\xi_t$, where the conditional mean $b_t$ is predictable and has norm at most $M$. Assume the processed noise has total second moment at most $V$ and directional $p$th moments at most $s_0^p$. The raw oracle need not satisfy this second-moment assumption. For decisions in a ball of radius $R$, \cref{lem:online} proves that every fixed comparator $u$ satisfies, with probability at least $1-\alpha$,
\begin{equation}\label{eq:overview-online}
 \frac1n\sum_{t=1}^n\ip{b_t}{z_t-u}
 \le C_pR\left[\frac{M+\sqrt V}{\sqrt n}
 +\frac{M\log(3/\alpha)}n
 +s_0\left(\frac{\log(3/\alpha)}n\right)^q\right].
\end{equation}
The estimate allows changing means and conditional noise laws. Its distinguishing feature is that $\sqrt V$ does not multiply the leading confidence factor.

To obtain this estimate, the learner maintains mixtures of Gaussian densities over virtual decisions. If $U=Z+aN$, with $Z$ supported on the decision ball and $N$ an independent standard Gaussian, the noise splits as
\[
 \ip\xi{U-x}=\ip\xi{Z-x}+a\ip\xi N.
\]
The first direction is bounded and uses the directional $p$th moment. For the second, $\E_N|\ip\xi N|^2=\norm\xi^2$, so only the total second moment enters. A scalar clipped exponential estimate combines these controls as
\[
 \E\exp(\pm E(U,\xi))
 \le\exp\{C_p[\eta^p(s_0R)^p+\eta^2a^2(V+M^2)]\},
\]
where~$E(U,\xi)=\chi(\eta\ip{b+\xi}{U-x})-\eta\ip b{U-x}$ and $\chi$ clips to $[-1,1]$. Entropy mirror updates then compare the learner to a Gaussian centered at the comparator. The initial divergence is at most $R^2/(2a^2)$, with no factor of $d$. This is why an exponentially large internal representation can yield a dimension-independent oracle bound. The construction is proved in \cref{app:online}, including the parameter choices.

\subsection{Conditional rejection and feasible proximal queries}
To use \cref{eq:overview-online}, we must create a finite second moment without losing the weak moment scale. At a frozen query, let $Y=b+\zeta$ be a corrected observation with $\norm b\le4G$. Accept it when $\norm Y\le\tau$, where $\tau\ge64(4G+L)$. If $Q$ is the accepted law, then \cref{lem:accept} shows
\[
\begin{array}{@{}>{\displaystyle}l@{\qquad}>{\displaystyle}l@{}}
 \Pp\{\text{reject}\}\le1/16,
 &\norm{\E_QY-b}\le C_ps(L/\tau)^{p-1},\\[6pt]
 \E_Q\norm{Y-\E_QY}^2\le C_pL^p\tau^{2-p},
 &\sup_{\norm v=1}\E_Q|\ip v{Y-\E_QY}|^p\le C_ps^p.
\end{array}
\]
Centering is essential for the bias bound: the rejected directional mean is bounded by H\"older's inequality using $s$, even though the rejection event is defined by the norm. Combining these estimates with the online inequality gives the error amplitude
\begin{equation}\label{eq:overview-amplitude}
 \frac{G+\sqrt{L^p\tau^{2-p}}}{\sqrt n}
 +\frac{G\log(3/\alpha)}n
 +s\left(\frac{\log(3/\alpha)}n\right)^q
 +s(L/\tau)^{p-1}.
\end{equation}
There is no threshold-times-confidence term $\tau\log(3/\alpha)/n$. Taking $\tau$ proportional to $\max\{4G+L,L(T/e)^{1/p}\}$ balances the new second moment and the bias and yields $v_T$.

The corrected observation must also be available under the stated composite oracle. Given a virtual point $z$ in a ball containing $x^*$, a short search over proximal parameters produces a feasible point $y$, a nearby relative-interior query $w$, and a bounded predictable correction $a$. The transformed observation $Y=\Pi g(w,\xi)+a$, where $\Pi$ projects onto the affine hull's translation space, has bounded mean and satisfies
\[
 F(y)-F^*\le\ip{\E Y}{z-x^*}+\frac{GR}{2n}.
\]
Projection and predictable correction preserve the centered moment bounds. Thus averaging the feasible points transfers online regret to objective error without assuming a proximal map for an additional ball constraint.

Each accepted round may require several raw calls. The procedure permits at most $4n$ calls for $n$ accepted rounds; conditional rejection probability at most $1/16$ makes exhaustion exponentially unlikely. The analysis couples this capped procedure to an ideal accepted process instead of conditioning concentration on eventual completion. The resulting convex primitive is proved in \cref{sec:rejection}.

\subsection{Curvature and confidence allocation}
For $\mu=0$, one invocation of the primitive gives the upper bound, together with a zero-query proximal cap of order $GD$. For $\mu>0$, strong convexity implies $F(x)-F^*\ge\mu\norm{x-x^*}^2/2$. An error guarantee of order $\mu R^2$ therefore certifies a smaller radius, allowing the primitive to be restarted at successively halved radii.

If there are $J$ stages, set $R_j=2^{-j}R_0$ and $\delta_j=\delta\,2^{j-J-1}$. Then $\log(3/\delta_j)=\ell+(J-j+1)\log2$. Earlier stages receive smaller failure probabilities but have geometrically smaller sample costs. The resulting geometric sums absorb the extra $J-j$ factor without a horizon logarithm. Combining the localized guarantee, the unlocalized convex guarantee, and zero-query preprocessing gives the envelope $\cH$. All choices are made from public parameters; selecting among these procedures does not require observing the unknown objective value. Details appear in \cref{sec:upper-full}.

\section{Lower Bound via Adaptive Testing}\label{sec:lower-overview}

The leading heavy-tail and confidence orders in this section are already present in the lower bounds of \citet{liu2026}. We give a unified construction to retain the additional finite-budget Gaussian term and the initialization-dependent curvature envelope. The construction converts hidden coordinate signs into objective error in both curvature regimes. For $k\le d$, $v\in\{-1,1\}^k$, and $0<a\le G$, define
\[
 f_v(x)=\frac a{\sqrt k}\sum_{j=1}^k|x_j-\rho v_j|,
 \qquad r(x)=\frac\mu2\norm x^2,
 \qquad \rho=\frac{R_a}{\sqrt k},
\]
where~$R_a=D$ if $\mu=0$ and $R_a=\min\{D,a/\mu\}$ otherwise. The minimizer is $(\rho v,0)$, within distance $D$ of the origin. A wrong output sign in coordinate $j$ costs at least $\cH(a)/k$. Thus the same loss-to-testing reduction yields a linear cost $Da$ for convex objectives and a quadratic cost $a^2/\mu$ in the strongly convex quadratic branch.

Three oracle constructions force complementary choices of $a$. Gaussian noise in $k=\lceil e\rceil$ coordinates yields $a$ of order $\min\{G,L/\sqrt T\}$: neighboring sign vectors have small conditional KL divergence at every query, so the chain rule controls the entire adaptive transcript. Independent Bernoulli masks, each revealing a coordinate with probability of order $1/T$, yield $a$ of order $\min\{G,A/T^q\}$. A constant fraction of signs remain completely hidden, even when query locations are adaptive. Finally, a one-coordinate erasure oracle makes all observations vanish with probability greater than $\delta$ under two different objectives. Their identical transcripts force the amplitude $\min\{G,s(\ell/T)^q\}$.

Each construction satisfies both centered moment bounds and uses the same known regularizer. Taking their maximum and applying the scaling of $\cH$ gives the lower side of \cref{thm:main}. The dimension condition is independent of $T$, and the proof includes internal randomization and arbitrary feasible outputs. \Cref{sec:lower-poly} verifies membership, transcript information, and the conversion to the minimax quantile.

\section{Related Work}\label{sec:related}

\paragraph{Convex optimization under moment assumptions.}
Classical oracle complexity separates information obtained from a function from the computation performed between queries \citep{nemirovskiyudin1983}; \citet{lan2020} develops the first-order and stochastic optimization framework. For infinite-variance noise, \citet{vural2022} analyze mirror descent and establish optimal stochastic convex rates under their moment assumptions. \citet{liuzhou2023} obtain high-probability, expected-error, and initial-distance-adaptive guarantees for nonsmooth convex optimization using clipped gradients. These works establish the underlying heavy-tail budget exponents. Our question concerns the dependence on two distinct centered moment scales and the confidence level. More recently, \citet{liuoco2026} establish optimal expected regret for classical online methods on bounded domains; this expectation guarantee does not determine the high-probability confidence cost.

\paragraph{Clipping and high-probability guarantees.}
\citet{zhangcutkosky2022} develop parameter-free high-probability regret bounds for heavy-tailed online convex optimization on unbounded domains. \citet{nguyen2023} refine high-probability analyses of clipped gradient methods, while \citet{parletta2024} study clipped subgradient schemes under finite variance. \citet{parletta2025} improve the anytime dependence for averaged clipped subgradient iterates on possibly unbounded domains and establish last-iterate rates in expectation. These guarantees do not separate the two moment scales considered here. \citet{gorbunov2024} address nonsmooth stochastic optimization with heavy-tailed noise; \citet{sadiev2023} cover smooth optimization and variational inequalities with bounded centered moments below order two. Recent work of \citet{ouwang2026} studies smooth Polyak--\L{}ojasiewicz optimization under a centered norm-moment bound, including dimension-dependent lower bounds. These settings clarify the roles of domain assumptions, smoothness, and moment order. The present online primitive instead separates a weak directional $p$th moment from a total second moment created by preprocessing.

\paragraph{Refined noise geometry and distributional structure.}
Weak directional moments also arise in robust mean estimation: \citet{cherapanamjeri2022} characterize optimal estimation without a finite variance for independent observations from a fixed distribution. \citet{tang2022} use robust estimation in smooth projection-free stochastic optimization. Neither result directly supplies an online guarantee for the changing subgradient means of our nonsmooth composite model.
\citet{marsden2025} study isotropic high-probability directional control and sub-exponential noise, as well as quantum sampling access. These stronger tail or different oracle assumptions distinguish their dimension-dependent guarantees from our centered $p$th-moment model.
\citet{das2024} obtain refined guarantees for streaming heavy-tailed estimation using clipped SGD, including improvements in the finite-variance setting. The closest comparison is \citet{liu2026}, which uses generalized effective dimension in nonsmooth composite optimization and supplies corresponding lower bounds; the quantitative gap is stated in \cref{sec:introduction}. Our upper algorithm attains the smaller confidence coefficient with unrestricted computation. A different direction is pursued by \citet{puchkin2024}, who use structured noise densities and smoothed median estimates to obtain faster convergence. Such additional distributional assumptions are not part of our two-moment class.

\paragraph{Information-theoretic optimization lower bounds.}
\citet{raginskyrakhlin2011} formulate black-box optimization through feedback information, and \citet{agarwal2012} develop information-theoretic lower bounds for stochastic convex optimization. Our proof follows this adaptive-transcript perspective. Under the two moment constraints, \citet{liu2026} already establish the leading heavy-tail statistical lower bounds. Our Gaussian, masked-coordinate, and erasure families recover these orders within a unified finite-budget envelope. Their nonlinear losses provide the required conversion from hidden information to objective error without treating subgradients at changing query points as repeated observations of a single mean.

\section{Discussion}\label{sec:discussion}
The two curvature regimes share the statistical amplitude $L/\sqrt T+s((e+\ell)/T)^q$. The moment ratio determines its confidence-independent part, while the weak directional scale determines its confidence cost. The envelope $\cH_{\mu,D}$ then captures the effect of curvature, initial distance, and saturation. This separates the statistical question from the optimization reduction and gives matching statistical rates throughout the stated curvature regimes.

The main unresolved issue is computational: the proof uses a large Gaussian-mixture family rather than an efficient clipped-gradient update. An efficient method with the same two-scale confidence dependence would strengthen the algorithmic interpretation. Parameter adaptation and the deterministic finite-dimensional terms are also beyond the present characterization.

\begingroup
\small
\setlength{\bibsep}{4pt}
\bibliographystyle{unsrtnat}
\bibliography{references}
\endgroup
\clearpage
\appendix
\section{The Common Online and Convex Optimization Primitive}\label{app:online}
This section proves the shared statistical ingredient. We first separate directional and Gaussian perturbations in an exponential estimate, then use that estimate in an online learner and transfer its regret to feasible composite optimization.

\subsection{A scalar clipping inequality}
Let $\chi(x)=\max\{-1,\min\{x,1\}\}$.
\begin{lemma}\label[lemma]{lem:mixed-scalar}
For real $t,u,v$, $\varepsilon\in\{-1,1\}$, and $1<p\le2$,
\begin{equation}\label{eq:mixed-scalar}
 e^{\varepsilon[\chi(t+u+v)-t]}
 \le1+\varepsilon(u+v)
 +C e^{2|t|}\left(|u|^p+v^2+t^2\ind_{\{|t|>1/2\}}\right),
\end{equation}
where an absolute constant such as $C=256$ suffices.
\end{lemma}
\begin{proof}
First suppose $|t|\le1/2$ and write $z=u+v$. If $|z|\le1/2$, clipping does nothing and the excess of the left side over $1+\varepsilon z$ is at most $Cz^2$. If $|z|>1/2$, that excess is at most $e^{3/2}+1+|z|$, and hence at most $C\min\{z^2,|z|\}$. The same latter bound holds in the first case. Let $\rho(z)=\min\{z^2,|z|\}$. At least one of $|u|,|v|$ is at least $|z|/2$. If it is $|u|$, then
$\rho(z)\le4\rho(u)\le4|u|^p$; otherwise $\rho(z)\le z^2\le4v^2$. This proves the assertion for $|t|\le1/2$.

If $|t|>1/2$, the left side is at most $e^{1+|t|}$. Also
$|u+v|\le |u|^p+v^2+5/4$, while $t^2>1/4$ and
$e^{1+|t|}\le C e^{2|t|}t^2$ in this range. The remaining terms in \eqref{eq:mixed-scalar} therefore dominate both the possible negative linear term and this exponential bound. This proves the lemma.
\end{proof}

\subsection{A conditional Gaussian-mixture estimate}
Work in a finite-dimensional Euclidean space. For a ball $\mathcal B=B(c,R)$, let $U=Z+aN$, where $Z$ is supported on $\mathcal B$, $N$ is standard Gaussian, and they are independent. These are auxiliary integration variables, not oracle assumptions or stochastic calls.

\begin{lemma}\label[lemma]{lem:mixed-exp}
Suppose $x\in\mathcal B$, $\norm b\le M$, and a centered random vector $\xi$ satisfies
\begin{equation}\label{eq:mixed-assumptions}
 \E\norm\xi^2\le V,\qquad
 \sup_{\norm v=1}\E|\ip v\xi|^p\le s_0^p.
\end{equation}
If $U$ is independent of $\xi$, $0<a\le R$, and $\eta MR\le1/16$, then, for
$E(U,\xi)=\chi(\eta\ip{b+\xi}{U-x})-\eta\ip b{U-x}$,
\begin{equation}\label{eq:mixed-exp}
 \E e^{\pm E(U,\xi)}\le e^{\upsilon},\qquad
 \upsilon=C_p\{\eta^p(s_0R)^p+\eta^2a^2(V+M^2)\}.
\end{equation}
The assertion holds conditionally on an arbitrary past under the same conditional assumptions.
\end{lemma}
\begin{proof}
Condition on $Z$ and put $v_0=Z-x$, so $\norm{v_0}\le2R$. In Lemma~\ref{lem:mixed-scalar} use
\[
 t=\eta\ip b{v_0+aN},\qquad
 u=\eta\ip\xi{v_0},\qquad v=\eta a\ip\xi N.
\]
For fixed $N$, the expected linear term is zero. Since
$e^{2|t|}\le e^{1/4}e^{2\eta a|\ip bN|}$ and $N$ is independent of $\xi$, the contribution of $|u|^p$ is at most $C_p\eta^p(s_0R)^p$.
For the quadratic contribution condition on $\xi$ and use Cauchy--Schwarz:
\[
 \E_N[e^{2\eta a|\ip bN|}|\ip\xi N|^2]
 \le(\E_Ne^{4\eta a|\ip bN|})^{1/2}
     (\E_N|\ip\xi N|^4)^{1/2}
 \le C\norm\xi^2.
\]
Only one-dimensional Gaussian moments occur. Thus this contribution is at most $C\eta^2a^2 V$.
Finally write $t=t_0+b_0N_1$, where $|t_0|\le1/8$ and $|b_0|\le\eta aM\le1/16$. On $|t|>1/2$ one has $|b_0N_1|>3/8$ and $|t|\le2|b_0N_1|$. Therefore
\[
 \E[e^{2|t|}t^2\ind_{\{|t|>1/2\}}]
 \le Cb_0^2\E[N_1^2e^{2|b_0N_1|}]
 \le C\eta^2a^2M^2.
\]
Average over $Z$ and use $1+\upsilon\le e^\upsilon$. Each step is valid for conditional distributions.
\end{proof}

\subsection{Known-density mirror updates}
At round $t$ choose $z_t\in\mathcal B$ and observe $Y_t=b_t+\xi_t$, where $b_t$ is predictable, $\norm{b_t}\le M$, and \eqref{eq:mixed-assumptions} holds conditionally. Fix a horizon $n$, a finite $R/(8n)$-net $\mathcal Z\subseteq\mathcal B$ containing $c$, and the convex family of Gaussian-mixture densities
\[
 \mathcal C_a=\left\{\sum_{z\in\mathcal Z}w_zN(z,a^2I):w\in\Delta_{|\mathcal Z|}\right\}.
\]
Let $m(P)=\E_PU$ and
\[
 \mathscr D(P\Vert Q)=\KL(P\Vert Q)+\frac\beta2\norm{m(P)-m(Q)}^2.
\]
This is a Bregman divergence of density entropy plus a quadratic of the mean. The KL is between densities, not between mixing weights. Initialize $P_1=N(c,a^2I)$, set $z_t=m(P_t)$, and after the one observation define
\[
 \Psi_t(u)=\chi(\eta\ip{Y_t}{u-z_t}),\qquad
 P_{t+1}\in\arg\min_{P\in\mathcal C_a}
       \{\E_P\Psi_t+\mathscr D(P\Vert P_t)\}.
\]
The update depends only on known densities and the observed vector.

\begin{lemma}[Conditional mixed-moment online rate]\label[lemma]{lem:online}
For $M>0$, $V\ge0$, $s_0>0$, and $0<\alpha<1/10$, choose the parameters in \cref{eq:tuning}. For any fixed comparator $u\in\mathcal B$, with probability at least $1-\alpha$,
\begin{equation}\label{eq:online}
 \frac1n\sum_{t=1}^n\ip{b_t}{z_t-u}
 \le C_pR\left[\frac{M+\sqrt V}{\sqrt n}
             +\frac{M\ell_\alpha}{n}
             +s_0(\ell_\alpha/n)^q\right],
 \qquad \ell_\alpha=\log(3/\alpha).
\end{equation}
The result is conditional on any history at the beginning of a run, and permits arbitrarily changing predictable means and conditional noise laws.
\end{lemma}
\begin{proof}
For a fixed grid comparator $u_*$ and $Q=N(u_*,a^2I)$, the Bregman three-point identity and update optimality give
\[
 -\sum_t\E_Q\Psi_t\le\mathscr D(Q\Vert P_1)+\sum_t\mathcal V_t,
 \qquad \mathcal V_t=\sup_{P\in\mathcal C_a}
                  \{-\E_P\Psi_t-\mathscr D(P\Vert P_t)\}.
\]
This follows by summing
$\mathscr D(Q\Vert P_{t+1})\le\mathscr D(Q\Vert P_t)+\E_Q\Psi_t+\mathcal V_t$ and dropping the nonnegative final divergence.
For analysis set $E_t(u)=\Psi_t(u)-\eta\ip{b_t}{u-z_t}$ and
$W_t=\E_{P_t}e^{-E_t(U)}$.
Completing a square in the mean and applying the entropy variational inequality yields
\[
 \mathcal V_t\le\eta^2M^2/(2\beta)+\log W_t.
\]
Indeed $-\eta\ip{b_t}{m(P)-z_t}-(\beta/2)\norm{m(P)-z_t}^2\le\eta^2M^2/(2\beta)$, and Jensen under~$P$~gives
$-\E_PE_t-\KL(P\Vert P_t)\le\log\E_{P_t}e^{-E_t}$.

By Lemma~\ref{lem:mixed-exp}, $\E_{t-1}W_t\le e^\upsilon$ and, by Jensen,
$\E_{t-1}e^{\E_QE_t}\le e^\upsilon$. The two nonnegative product supermartingales give, each with failure at most $\alpha/2$,
\[
 \sum_t\log W_t\le n\upsilon+\log(2/\alpha),\qquad
 \sum_t\E_QE_t\le n\upsilon+\log(2/\alpha).
\]
No union bound over the grid is used. Since
$\eta\sum_t\ip{b_t}{z_t-u_*}=-\sum_t\E_Q\Psi_t+\sum_t\E_QE_t$, these~imply
\begin{equation}\label{eq:master}
 \eta\sum_t\ip{b_t}{z_t-u_*}\le
 \frac{R^2}{2a^2}+\frac{\beta R^2}{2}
 +\frac{n\eta^2M^2}{2\beta}+2n\upsilon+2\log(2/\alpha).
\end{equation}
The Gaussian comparison cost $R^2/(2a^2)$ has no ambient-dimension factor.
Take
\begin{equation}\label{eq:tuning}
 \begin{split}
 a&=R/\sqrt{\ell_\alpha},\qquad \beta=\eta M\sqrt n/R,\\
 \eta&=\min\left\{\frac1{16MR},
 \frac{\ell_\alpha}{R\sqrt{n(V+M^2)}},
 \frac1{s_0R}(\ell_\alpha/n)^{1/p}\right\}.
 \end{split}
\end{equation}
Divide \eqref{eq:master} by $\eta n$. The two $\beta$ terms sum to $MR/\sqrt n$. The other terms are bounded by
\[
 C_p\left[\frac{\ell_\alpha}{\eta n}
   +\eta^{p-1}(s_0R)^p
   +\frac{\eta R^2(V+M^2)}{\ell_\alpha}\right].
\]
To check the balance explicitly, the reciprocal of the minimum defining $\eta$ gives
\[
 \frac{\ell_\alpha}{\eta n}
 \le R\left[\frac{16M\ell_\alpha}{n}
       +\sqrt{\frac{V+M^2}{n}}
       +s_0(\ell_\alpha/n)^q\right].
\]
The third choice bounds $\eta^{p-1}(s_0R)^p$ by $s_0R(\ell_\alpha/n)^q$, and the second bounds $\eta R^2(V+M^2)/\ell_\alpha$ by $R\sqrt{(V+M^2)/n}$. These are the terms in \eqref{eq:online}. Approximating $u$ by $u_*$ within $R/(8n)$ adds at most $MR/(8n)$. This is absorbed. The comparator is fixed after conditioning on the initial history, which proves the conditional assertion.
\end{proof}

\paragraph{Computational interpretation.}
The mixture objective is continuous on a compact simplex and strictly convex in the density; all component densities are positive, so it has a measurable minimizer. Its entropy integrals and their derivatives have finite uniform Gaussian-tail bounds even on the boundary of the simplex: a mixture is at least the minimum of its finitely many component densities, and ratios of components are exponentials of linear functions. Deterministic numerical integration and finite optimization to first-order residual $\eta MR/(8n)$ are possible with unrestricted internal computation. This residual adds only $MR/(8n)$ to the average regret. The net can be exponentially large in dimension, but these deterministic computations require no additional oracle calls.

\section{Feasible Queries and Conditional Rejection}\label{sec:rejection}
The online result requires bounded conditional means and a finite trace second moment. We now construct observations with these properties from the original composite oracle. The construction does not use strong convexity.
\subsection{Feasible queries and bounded predictable corrections}
We give the reduction explicitly so that no proximal map of a newly restricted domain is assumed. Work in the affine hull of $\mathcal X$, with translation space $V_0$ and projection $\Pi_{V_0}$. If it is zero-dimensional there is nothing to optimize. Fix an anchor $x^\circ\in\operatorname{ri}\mathcal X$. At any relative-interior point $w$, the tangent component of every subgradient of $f$ has norm at most $G$: apply the subgradient inequality in both directions of a sufficiently short feasible segment through $w$.

Suppose a feasible center $c$ has $\norm{c-x^*}\le R$ and $z\in B_{V_0}(c,R)$ is a virtual point. Minimality~gives
\begin{equation}\label{eq:h-lower}
 h(u)\ge h(x^*)-G\norm{u-x^*}.
\end{equation}
For $\lambda>0$ put
$y_\lambda=\prox_{\lambda h}(z)$ and $a_\lambda=(z-y_\lambda)/\lambda\in\partial h(y_\lambda)$.
Start at $\lambda_0=2R/G$. Comparison to $x^*$ in the proximal objective and \eqref{eq:h-lower} imply, with $t=\norm{y_\lambda-z}$ and $b_0=\norm{z-x^*}\le2R$,
$t^2\le b_0^2+2\lambda G(t+b_0)$, hence $t\le2\lambda G+b_0$. Thus $\norm{a_{\lambda_0}}\le3G$.
If $\norm{a_\lambda}\ge G$, stop; otherwise halve $\lambda$ until either this holds or $\lambda\le R/(Gn)$. Proximal displacement is nondecreasing in the parameter (compare the two proximal objectives and add), so $\norm{a_{\lambda/2}}\le2\norm{a_\lambda}$. At termination,
\begin{equation}\label{eq:correction}
 \norm a\le3G,\qquad a=(z-y)/\lambda\in\partial h(y),
 \qquad \norm a\ge G\ \text{or}\ \lambda\le R/(Gn).
\end{equation}
Only $1+\lceil\log_2(2n)\rceil$ free proximal evaluations are needed.
Choose $w\in\operatorname{ri}\mathcal X$ on the segment from $y$ to $x^\circ$ with $\norm{w-y}\le R/(8n)$. A raw observation is transformed into
\begin{equation}\label{eq:corrected-observation}
 Y=\Pi_{V_0}g(w,\xi)+a,\qquad
 b=\E Y=\bar g+a,\quad \norm b\le4G.
\end{equation}
Its centered noise obeys \eqref{eq:native}. The subgradient inequalities at $w$ and $y$ give
\[
 F(y)-F^*\le\ip{\bar g+a}{y-x^*}+2G\norm{w-y}.
\]
Since $z-y=\lambda a$,
$-\ip{\bar g+a}{z-y}\le\lambda(G\norm a-\norm a^2)$, which is nonpositive or at most $GR/(4n)$ by \eqref{eq:correction}. Consequently
\begin{equation}\label{eq:transfer}
 F(y)-F^*\le\ip b{z-x^*}+GR/(2n).
\end{equation}
Everything selecting $w$ is determined before any raw observation at that virtual round. The set restricting virtual centers is not a new constraint on actual proximal evaluations.

\subsection{Moments of accepted observations}
\begin{lemma}\label[lemma]{lem:accept}
For $Y=b+\zeta$, $\norm b\le4G$, satisfying \eqref{eq:native}, fix $\tau\ge64(4G+L)$. Let $Q$ be its conditional law given $\norm Y\le\tau$, and $b'=\E_QY$. Then
\begin{align}
 \Pp\{\norm Y>\tau\}&\le(2L/\tau)^p\le1/16,\label{eq:accept-prob}\\
 \norm{b'-b}&\le C_ps(L/\tau)^{p-1}=:B_\tau,\label{eq:accept-bias}\\
 \E_Q\norm{Y-b'}^2&\le C_pL^p\tau^{2-p}=:V_\tau,\label{eq:accept-V}\\
 \sup_{\norm v=1}\E_Q|\ip v{Y-b'}|^p&\le C_ps^p.\label{eq:accept-weak}
\end{align}
All assertions also hold conditionally at a frozen adaptive query point.
\end{lemma}
\begin{proof}
Rejection implies $\norm\zeta>\tau/2$, proving \eqref{eq:accept-prob}. Writing the rejection probability as $a_0$, centering and H\"older's inequality give
\[
 |\ip v{b'-b}|=\frac{|\E[\ip v\zeta\,\ind_{\{\norm Y>\tau\}}]|}{1-a_0}
 \le\frac{s a_0^q}{1-a_0}.
\]
This proves \eqref{eq:accept-bias} by duality. There is no mean-times-rejection-probability bias term. On acceptance $\norm\zeta\le3\tau/2$, so
$\E_Q\norm\zeta^2\le (3\tau/2)^{2-p}L^p/(1-a_0)$; subtracting the conditional mean decreases the second moment. Finally, centering and Jensen~imply
$\E_Q|\ip v{Y-b'}|^p\le2^p\E_Q|\ip v\zeta|^p\le2^ps^p/(1-a_0)$.
\end{proof}

\begin{proposition}[Hybrid conditional convex primitive]\label[proposition]{prop:primitive}
Fix $\tau\ge64(4G+L)$. For $n\ge C_p\ell_\alpha$, $\ell_\alpha=\log(3/\alpha)$, and a conditionally valid radius $\norm{c-x^*}\le R$, at most $4n$ raw calls suffice for
\begin{equation}\label{eq:primitive}
 F(\widehat x)-F^*\le C_pR\left[
 \frac{G+\sqrt{L^p\tau^{2-p}}}{\sqrt n}
 +\frac{G\ell_\alpha}{n}
 +s(\ell_\alpha/n)^q+s(L/\tau)^{p-1}\right]
\end{equation}
with conditional probability at least $1-\alpha$.
\end{proposition}
\begin{proof}
At each virtual round freeze the query and correction \eqref{eq:corrected-observation}. Draw fresh copies until $\norm Y\le\tau$, then use that observation in the density mirror algorithm analyzed in Lemma~\ref{lem:online}. Ignore rejected values for update selection. In the ideal run with no cap, the accepted conditional law is $Q_t$; its mean has norm at most $4G+B_\tau$, its centered second moment is at most $V_\tau$, and its weak $p$th moment is at most $C_ps^p$.
Apply Lemma~\ref{lem:online} with failure level $\alpha/2$. Replacing accepted means $b'_t$ by original means $b_t$ in the regret costs at most $2RnB_\tau$. The extra occurrences of $B_\tau/\sqrt n$ and $B_\tau\ell_\alpha/n$ are bounded by a constant times $B_\tau$ because $n\ge C_p\ell_\alpha$. Averaging \eqref{eq:transfer} proves \eqref{eq:primitive} for the ideal run.

Cap the run at $4n$ raw trials. Failure to obtain $n$ acceptances requires at least $3n$ rejections among those trials. For any prescribed $3n$ trial times, iterated conditioning bounds the probability of rejection at all of them by $16^{-3n}$. A union bound therefore gives
\begin{equation}\label{eq:raw-count}
 \Pp\{\text{budget exhausted before $n$ acceptances}\}
 \le2^{4n}16^{-3n}=2^{-8n}\le\alpha/2.
\end{equation}
One may extend the ideal run with dummy trials to define the event after its $n$th acceptance. Couple the capped algorithm to the ideal run, rather than conditioning the online concentration proof on completion. This combines the two failure probabilities correctly. The actual output is the average of the proximal points associated with accepted rounds; all belong to the effective domain of $h$. On cap exhaustion, return a predetermined feasible proximal point; the failure event already includes this case. There are no uncharged replacement samples.
\end{proof}

\section{Upper Bounds and Curvature Localization}\label{sec:upper-full}
The common primitive gives a linear-in-radius error bound. We first use it directly for $\mu=0$, then exploit strong convexity through certified radii. Both arguments use the same threshold, and we optimize that threshold only after establishing their common fixed-threshold form.

\subsection{The general convex branch}
For a fixed $\tau\ge64(4G+L)$, set
\begin{equation}\label{eq:v-tau}
 v_T(\tau)=\frac{G+\sqrt{L^p\tau^{2-p}}}{\sqrt T}
 +\frac{G\ell}{T}+s(\ell/T)^q+s(L/\tau)^{p-1}.
\end{equation}
A free proximal evaluation provides the cap $GD$. Indeed, for $y=\prox_{\eta h}(x_1)$, proximal optimality and Lipschitzness give
\[
 F(y)-F^*\le\frac{D^2}{2\eta}+\frac{\eta G^2}{2}.
\]
Choosing $\eta=D/G$ proves the cap without any stochastic observation. If $T<C_p\ell$, the cap is sufficient because $G\ell/T$ is a constant fraction of $G$. Otherwise, use \cref{prop:primitive} with $n=\lfloor T/4\rfloor$, radius $D$, and failure level $\delta$. Taking the better parameter-only guarantee yields
\begin{equation}\label{eq:convex-fixed}
 \cE_T^{(0)}(\delta)\le C_pD\min\{G,v_T(\tau)\}
 =C_p\cH_{0,D}(\min\{G,v_T(\tau)\}).
\end{equation}
This argument also gives the bound $C_pDv_T(\tau)$ when $\mu>0$.

\subsection{Strong convexity and zero-query preprocessing}
In this subsection and the following localization argument, assume $\mu>0$.
For $y=\prox_{\eta h}(x_1)$, strong convexity of the proximal objective gives
\[
 h(y)-h(x^*)\le\frac{D^2-\norm{y-x_1}^2}{2\eta}
                   -\frac{\mu+\eta^{-1}}2\norm{y-x^*}^2.
\]
Adding Lipschitzness and maximizing the distance quadratic yields
\begin{equation}\label{eq:zero}
 F(y)-F^*\le D^2/(2\eta)+G^2/[2(\mu+\eta^{-1})].
\end{equation}
Minimizing in $\eta^{-1}\ge0$ gives the infimum zero-query envelope $\cH(G)$. Finite parameters attain it within factor two: take $\eta=D/G$ for $D\le G/\mu$, and $\eta=\mu D^2/G^2$ otherwise. The matching lower bound is given by $G|x_1\mp\rho|$, $\rho=\min\{D,G/\mu\}$, with the common quadratic regularizer: one of the two instances has wrong-half-line probability at least $1/2$ without observations.

A finite warm start gives a center $c_0$ with radius
\begin{equation}\label{eq:warm}
 R_0=\begin{cases}D,&D\le G/\mu,\\ \sqrt2G/\mu,&D>G/\mu,\end{cases}
 \qquad \norm{c_0-x^*}\le R_0\le\sqrt2G/\mu.
\end{equation}
Use $c_0=x_1$ in the first case and the second finite proximal parameter above in the other. Then \eqref{eq:zero} gives $F(c_0)-F^*\le G^2/\mu$, and strong convexity implies
\begin{equation}\label{eq:growth}
 F(x)-F^*\ge(\mu/2)\norm{x-x^*}^2.
\end{equation}
This direction follows directly by applying strong convexity on the segment from $x^*$ to $x$ and sending its parameter to zero. Only this direction of the error-to-distance implication is needed.

\subsection{Fixed threshold, halved radii, and backwards confidence}
Keep $\tau$ fixed and use $v_T(\tau)$ from \cref{eq:v-tau}.
Prescribe a number of stages $J\ge1$, radii $R_j=2^{-j}R_0$, and
\[
 \delta_j=\delta\,2^{j-J-1},\qquad
 \ell_j=\log(3/\delta_j)=\ell+(J-j+1)\log2.
\]
Their sum is less than $\delta$. Stage $j$ uses the primitive with the previous output as center, and the deterministic number of accepted rounds
\begin{equation}\label{eq:phase-budget}
 n_j=\left\lceil C_p\left[
 \frac{G^2+L^p\tau^{2-p}}{\mu^2R_{j-1}^2}
 +\frac{G\ell_j}{\mu R_{j-1}}
 +\ell_j\left(\frac{s}{\mu R_{j-1}}\right)^{r_p}
 \right]\right\rceil.
\end{equation}
Allow at most $4n_j$ raw calls. Since $G/(\mu R_{j-1})\ge1/\sqrt2$, the middle term ensures $n_j\ge C\ell_j$. If
\begin{equation}\label{eq:bias-target}
 B_\tau\le c_p\mu R_{J-1},
\end{equation}
then a sufficiently large budget constant makes the conditional primitive error at most $\mu R_{j-1}^2/8=\mu R_j^2/2$. Equation~\eqref{eq:growth} certifies the next radius. The probability of the first failed stage is at most $\delta_j$ conditional on all previous successes, so total failure is less than $\delta$.

\begin{lemma}\label[lemma]{lem:phase-sum}
The raw sample budget obeys
\begin{equation}\label{eq:phase-sum}
 4\sum_{j=1}^J n_j\le C_p\left[
 \frac{G^2+L^p\tau^{2-p}}{\mu^2R_J^2}
 +\frac{G\ell}{\mu R_J}
 +\ell\left(\frac{s}{\mu R_J}\right)^{r_p}\right].
\end{equation}
Its constant is independent of $J$.
\end{lemma}
\begin{proof}
Ceilings are absorbed by the first summand because $G^2/(\mu^2R_{j-1}^2)\ge1/2$. Write $m=J-j$. Then $R_{j-1}=2^{m+1}R_J$ and $\ell_j=\ell+(m+1)\log2$. For every $\gamma>0$,
\[
 \sum_{m=0}^{\infty}2^{-\gamma(m+1)}[\ell+(m+1)\log2]
 =\frac{2^{-\gamma}}{1-2^{-\gamma}}\ell
  +\frac{2^{-\gamma}}{(1-2^{-\gamma})^2}\log2
 \le C_\gamma\ell.
\]
Use exponents $1,r_p$ for the last two terms and the ordinary geometric series with exponent $2$ for the first. The factor four accounts for all raw observations, including rejected trials.
\end{proof}

\begin{proposition}\label[proposition]{prop:fixed-tau}
For $\mu>0$, every $T\ge1$, and every $\tau\ge64(4G+L)$,
\begin{equation}\label{eq:fixed-tau}
 \cE_T^{\rm sc}(\delta)\le
 C_p\cH(\min\{G,v_T(\tau)\}).
\end{equation}
\end{proposition}
\begin{proof}
For $T<C_p\ell$ the zero-query cap suffices, since $G\ell/T\gtrsim_p G$. Otherwise the primitive run with $n=\lfloor T/4\rfloor$ gives $C_pDv_T(\tau)$, with at most $T$ raw calls.
To get a squared bound warm start and put $r_*=M_pv_T(\tau)/\mu$, with $M_p$ sufficiently large. If $R_0<2r_*$, the same one-stage primitive from $c_0$ gives error at most $C_pR_0v_T(\tau)\le C_p'v_T(\tau)^2/\mu$.
Otherwise choose $J=\lfloor\log_2(R_0/r_*)\rfloor$, so $r_*\le R_J<2r_*$. The last term of \eqref{eq:v-tau} ensures \eqref{eq:bias-target}. Each of the three terms in \eqref{eq:phase-sum} is at most a constant times $T$ times, respectively, $M_p^{-2}$, $M_p^{-1}$, and $M_p^{-r_p}$. Increasing $M_p$ makes their sum at most $T$. The final stage gives error at most $\mu R_J^2/2\le C_pv_T(\tau)^2/\mu$.
All stage parameters are fixed before sampling. Choose just one of the zero-query, convex, and strong algorithms in advance using their parameter-only bounds. The minimum of $\cH(G)$, $C_pDv_T(\tau)$, and $C_pv_T(\tau)^2/\mu$ gives \eqref{eq:fixed-tau} by \eqref{eq:Hscale}. Finite proximal parameters attain the cap within a constant factor.
\end{proof}

\subsection{A common threshold and completion of the upper bound}
Set the rejection threshold to
\begin{equation}\label{eq:tau-choice}
 \tau=128\max\{4G+L,\ L(T/e)^{1/p}\}.
\end{equation}
It satisfies the required lower bound. If the first part of the maximum controls the trace term, then
\[
 \frac{\sqrt{L^p\tau^{2-p}}}{\sqrt T}\le C_p\frac{G+L}{\sqrt T}.
\]
The contribution of the second part is $C_p A/T^q$, since
$L(T/e)^{(2-p)/(2p)}T^{-1/2}=AT^{-q}$.
For $T\ge e$ the bias is at most $C_pA/T^q$; for $T<e$ it is at most $C_ps\le C_pL/\sqrt T$. Thus
\[
 v_T(\tau)\le C_p\left[\frac{G+L}{\sqrt T}
 +\frac{G\ell}{T}+\frac A{T^q}+s(\ell/T)^q\right]
 \le C_p'v_T.
\]
Combining \cref{eq:convex-fixed,prop:fixed-tau} proves the upper half of \cref{thm:main} for every $\mu\ge0$. Squaring and using $K^{2q}\asymp_p e^{2q}+\ell^{2q}$ gives \eqref{eq:expanded}. The weak-moment analysis avoids a $\tau\ell/T$ penalty in the accepted process.

\section{Adaptive Statistical Lower Bounds}\label{sec:lower-poly}
We use a single family for all $\mu\ge0$. Three observation laws force its hidden signs to remain uncertain, and the envelope converts sign errors into objective error. All use $r_\mu(x)=\mu\norm x^2/2$ (zero when $\mu=0$), $\mathcal X=\R^d$, and $x_1=0$.

\subsection{The nonlinear loss-to-testing reduction}
For $k\le d$, $0<a\le G$, and $v\in\{-1,1\}^k$, set
\[
 R_a=\begin{cases}D,&\mu=0,\\\min\{D,a/\mu\},&\mu>0,\end{cases}\qquad \rho=R_a/\sqrt k,
 \qquad f_v(x)=\frac a{\sqrt k}\sum_{j=1}^k|x_j-\rho v_j|.
\]
The function is convex and $a$-Lipschitz. Its regularized minimizer is $(\rho v,0)$, at distance $R_a\le D$. Decode the output coordinate signs, resolving zero as $-1$. Each incorrectly decoded sign has $x_jv_j\le0$ and causes loss at least
$a\rho/\sqrt k-\mu\rho^2/2=\cH(a)/k$. Consequently
\begin{equation}\label{eq:loss-hamming}
 F_v(x)-F_v^*\ge\frac{\cH(a)}k\Ham(\widehat v(x),v).
\end{equation}
If a Hamming error $J\in[0,k]$ has $\E J\ge7k/16$, then $\Pp\{J\ge k/4\}\ge1/4$, by comparing its expectation to $(k/4)(1-P)+kP$.

\subsection{Gaussian trace and heavy-tail components}
Let $k=\lceil e\rceil\le d$, $s_0=L/\sqrt k\le s$, and $a=\frac18\min\{G,L/\sqrt T\}$. Return a selected subgradient of $f_v$ plus $N(0,s_0^2I_k)$. Jensen gives the strong $p$th moment at most $L^p$, and every directional moment is at most $s_0^p\le s^p$.
Neighboring signs change the mean at any query point by at most $2a/\sqrt k$. The conditional KL per call is at most $2a^2/L^2$, so the adaptive transcript chain rule gives total KL at most $1/32$. Pinsker bounds total variation by $1/8$, giving each bit equal-prior error at least $7/16$. Under a uniform prior, \eqref{eq:loss-hamming} and the Hamming probability bound imply
\begin{equation}\label{eq:lb-trace}
 \cE_T^{(\mu)}(\delta)\ge c\cH(\min\{G,L/\sqrt T\}).
\end{equation}
The uniform sign prior is chosen independently of the algorithm seed. Averaging therefore supplies a fixed admissible instance with the stated failure probability for each randomized algorithm. Its entire internal randomization is included in the transcript. Algorithms that stop early can be padded to $T$ calls without changing their output.

For the heavy-tail term let $\alpha=1/(8T)$ and
$a=\frac1{16}\min\{G,Lk^{1/2-1/p}\alpha^q\}$. At every call take fresh independent $B_j\sim\operatorname{Bernoulli}(\alpha)$ and return
\[
 g_j(x)=\frac a{\sqrt k}\operatorname{sgn}(x_j-\rho v_j)\frac{B_j}{\alpha},
 \quad j\le k,
\]
with $\operatorname{sgn}(0)=0$ and zero other coordinates. Its centered coordinates obey
$\E|Z_j|^p\le2a^pk^{-p/2}\alpha^{1-p}$.
Independent centered scalar variables satisfy
$\E|\sum_jU_j|^p\le2^p\sum_j\E|U_j|^p$: symmetrize with independent copies, insert independent signs, and compare their $L_p$ norm with their $L_2$ norm. The resulting directional moment bound is
\[
 \E\norm Z^p\le2a^pk^{1-p/2}\alpha^{1-p}\le L^p,
 \qquad
 \E|\ip wZ|^p\le2^{p+1}a^pk^{1-p}\alpha^{1-p}\le s^p.
\]
The latter uses $\sum|w_j|^p\le k^{1-p/2}$ and $k\ge e$.
A coordinate is masked at all $T$ calls with probability at least $7/8$. Pre-generate masks and the random seed. On that event induction on the transcript shows complete independence from the corresponding unknown sign; the adaptive query itself is a function of earlier, sign-independent responses. Its conditional testing error is at least $1/2$. Thus expected Hamming error is at least $7k/16$, and
\begin{equation}\label{eq:lb-heavy}
 \cE_T^{(\mu)}(\delta)\ge c_p\cH(\min\{G,A/T^q\}).
\end{equation}

\subsection{A confidence component}
Put $b_0=\log(1/(4\delta))$, $u=e^{-b_0/T}$, and
$J_p(u)=u+u^p/(1-u)^{p-1}$. Use one-coordinate losses $a|x_1\mp R_a|$ with
$a=\min\{G,sJ_p(u)^{-1/p}\}$. Multiply the selected subgradient by $B/(1-u)$, where $B\sim\operatorname{Bernoulli}(1-u)$ is fresh. Its centered $p$th moment is at most $a^pJ_p(u)\le s^p$. The event that all Bernoulli masks are zero has probability $u^T=4\delta$; on this event all responses vanish and the adaptive transcripts coincide. One hypothesis has wrong-half-line probability at least $2\delta$, with direct loss at least $\cH(a)$.
Moreover
\[
 J_p(e^{-z})^{-1/p}\asymp_p
 \begin{cases}z^q,&0<z\le1,\\ e^{z/p},&z\ge1,\end{cases}
\]
by $1-e^{-z}\asymp z$ in the first range and $J_p(u)\asymp_p u$ in the second. In particular it is at least $c_pz^q$ for every positive $z$. Since $b_0\asymp\ell$ for $\delta<1/10$,
\begin{equation}\label{eq:lb-confidence}
 \cE_T^{(\mu)}(\delta)\ge c_p\cH(\min\{G,s(\ell/T)^q\}).
\end{equation}
Taking the maximum of \eqref{eq:lb-trace}--\eqref{eq:lb-confidence}, using \eqref{eq:Hscale}, and $K^q\asymp_p e^q+\ell^q$ proves the lower half of Theorem~\ref{thm:main} for all budgets and confidences. At atoms of the loss distribution, use a strictly smaller error threshold.

\section{Sample Complexity and Benchmark Calculations}\label{app:consequences}
These consequences invert the common amplitude; they do not require another optimization algorithm or lower-bound construction.

\subsection{Sample complexity through the curvature envelope}
Let $N^{(\mu)}(\epsilon,\delta)=\inf\{T\in\mathbb N_0:\cE_T^{(\mu)}(\delta)\le\epsilon\}$, and define $a_\epsilon=\cH_{\mu,D}^{-1}(\epsilon)$. The inverse is
\[
 a_\epsilon=
 \begin{cases}
 \epsilon/D,&\mu=0,\\
 \sqrt{2\mu\epsilon},&\mu>0,\ \epsilon\le\mu D^2/2,\\
 \epsilon/D+\mu D/2,&\mu>0,\ \epsilon\ge\mu D^2/2.
 \end{cases}
\]
For $0<\epsilon\le c_p\cH(G)$ with a sufficiently small constant, \cref{thm:main} and \cref{eq:Hscale} imply
\begin{align}
 N^{(\mu)}(\epsilon,\delta)
 &\le C_p\left[1+\frac{G^2+L^2}{a_\epsilon^2}
 +\frac{G\ell}{a_\epsilon}
 +(e+\ell)\left(\frac{s}{a_\epsilon}\right)^{r_p}\right],\label{eq:sample-upper}\\
 N^{(\mu)}(\epsilon,\delta)
 &\ge c_p\max\left\{\frac{L^2}{a_\epsilon^2},
 (e+\ell)\left(\frac{s}{a_\epsilon}\right)^{r_p}\right\},\label{eq:sample-lower}
\end{align}
where the lower bound assumes $d\ge\lceil e\rceil$. For the upper bound, make each term of $v_T$ at most a sufficiently small constant multiple of $a_\epsilon$, and then use the upper scaling inequality for $\cH$. For the lower bound, success in the stated nonsaturated range requires $w_T$ to be at most a $p$-dependent multiple of $a_\epsilon$. The trace, heavy-tail, and confidence constructions then force the two displayed costs; the maximum of the $e$ and $\ell$ costs is comparable to their sum. Restricting to small enough $\epsilon/\cH(G)$ also excludes a zero-query solution on the hard family. When the deterministic costs in \cref{eq:sample-upper} are dominated, the sample bounds match.

\subsection{The shared \texorpdfstring{$p=3/2$}{p=3/2} benchmark}
Under the parameters stated in \cref{sec:model}, $q=1/3$ and $\ell=e+\log3$, so
\[
 s(K/T)^q\asymp e^{-1/6}T^{-1/3}.
\]
The trace and deterministic $T^{-1/2}$ terms are dominated once $T\gtrsim e$. The confidence transient $e/T$ is dominated once $T\gtrsim e^{7/4}$. In that range the amplitude is below the saturation threshold, and for $\mu=1$ it also lies in the quadratic branch. The two error rates follow by applying $\cH_{0,1}(a)=a$ and $\cH_{1,1}(a)=a^2/2$ for $a\le1$.

Solving for $T$ gives
\[
 N^{(0)}(\epsilon,\exp(-e))=\Theta(e^{-1/2}\epsilon^{-3}),
 \qquad
 N^{(1)}(\epsilon,\exp(-e))=\Theta(e^{-1/2}\epsilon^{-3/2}).
\]
These sample-complexity statements hold, respectively, for $0<\epsilon\le c e^{-3/4}$ and $0<\epsilon\le c e^{-3/2}$ for a sufficiently small absolute $c$, so the corresponding budgets satisfy $T\ge C e^{7/4}$.
\end{document}